\documentclass[12pt]{amsart}

\usepackage{amsmath,amssymb,amsthm}
\usepackage[all]{xy}
\usepackage[dvipdfm,colorlinks=true]{hyperref}

\numberwithin{equation}{section}

\SelectTips{eu}{12}

\newtheorem{theorem}{Theorem}[section]
\newtheorem{proposition}[theorem]{Proposition}
\newtheorem{lemma}[theorem]{Lemma}
\newtheorem{corollary}[theorem]{Corollary}

\theoremstyle{definition}
\newtheorem{definition}[theorem]{Definition}

\newtheorem{problem}[theorem]{Problem}

\theoremstyle{remark}
\newtheorem{remark}[theorem]{Remark}

\newcommand{\Z}{\mathcal{Z}}
\newcommand{\RZ}{\mathbb{R}\mathcal{Z}}

\title[minimal Taylor resolutions]{Golodness and polyhedral products of simplicial complexes with minimal Taylor resolutions}

\author{Kouyemon Iriye}
\address{Department of Mathematics and Information Sciences, Osaka
Prefecture University, Sakai, 599-8531, Japan}
\email{kiriye@mi.s.osakafu-u.ac.jp}
\author{Daisuke Kishimoto}
\address{Department of Mathematics, Kyoto University, Kyoto, 606-8502, Japan}
\email{kishi@math.kyoto-u.ac.jp}
\thanks{K.I. is supported by JSPS KAKENHI (No. 26400094), and D.K. is supported by JSPS KAKENHI (No. 25400087)}


\begin{document}

\maketitle

\begin{abstract}
Let $K$ be a simplicial complex $K$ such that the Taylor resolution for its Stanley-Reisner ring is minimal. We prove that the following conditions are equivalent: (1) $K$ is Golod; (2) any two minimal non-faces of $K$ are not disjoint; (3) the moment-angle complex for $K$ is homotopy equivalent to a wedge of spheres; (4) the decomposition of the suspension of the polyhedral product $\Z_K(C\underline{X},\underline{X})$ due to Bahri, Bendersky, Cohen and Gitler desuspends. 
\end{abstract}

\baselineskip 17pt


\section{Introduction}

Golodness is a property of a graded commutative ring $R$ which is originally defined by a certain equality involving a Poincar\'e series of the cohomology of $R$, and Golod \cite{G} gave an equivalent condition in terms of the derived torsion algebra of $R$. Golodness has been intensively studied for Stanley-Reisner rings since those of important simplicial complexes such as dual sequentially Cohen-Macaulay complexes are known to have the Golod property, and in this paper, we are interested in Golodness of Stanley-Reisner rings. So we here define Golodness of Stanley-Reisner rings in terms of their derived torsion algebras. Let $K$ be a simplicial complex on the vertex set $[m]:=\{1,\ldots,m\}$, and let $\Bbbk$ be a commutative ring. Recall that the Stanley-Reisner ring of $K$ over $\Bbbk$ is defined by
$$\Bbbk[K]:=\Bbbk[v_1,\ldots,v_m]/(v_I\,\vert\,I\subset[m],\,I\not\in K)$$
where $|v_i|=2$ and $v_I=v_{i_1}\cdots v_{i_k}$ for $I=\{i_1,\ldots,i_k\}$. We consider the derived algebra $\mathrm{Tor}_{\Bbbk[v_1,\ldots,v_m]}(\Bbbk[K],\Bbbk)$ and fix its products and (higher) Massey products to those induced from the Koszul resolution of $\Bbbk$ over $\Bbbk[v_1,\ldots,v_m]$ tensored with $\Bbbk[K]$. Let $R^+$ denote the positive degree part of a graded ring $R$.

\begin{definition}
The Stanley-Reisner ring $\Bbbk[K]$ is called Golod if all products and (higher) Massey products in $\mathrm{Tor}^+_{\Bbbk[v_1,\ldots,v_m]}(\Bbbk[K],\Bbbk)$ are trivial.
\end{definition}

One of the biggest problem in Golodness of Stanley-Reisner rings is to get a combinatorial characterization of Golodness, where we have many examples of interesting simplicial complexes. This is still open at this moment while there have been many attempts. Then we consider the following weaker problem.

\begin{problem}
\label{problem}
Find a class of simplicial complexes for which Golodness of Stanley-Reisner rings can be combinatorially characterized.
\end{problem}

In a seminal paper \cite{DJ}, Davis and Januszkiewicz showed that the cohomology with coefficient $\Bbbk$ of a certain space constructed from a simplicial complex $K$, called the Davis-Januszkiewicz space for $K$, is isomorphic to the Stanley-Reisner ring $\Bbbk[K]$. This opens a way of a topological study of Stanley-Reisner rings. Moreover, Baskakov, Buchstaber and Panov \cite{BBP} found an isomorphism between the cohomology with coefficient $\Bbbk$ of the space $\Z_K$, called the moment-angle complex for $K$, and the derived torsion algebra $\mathrm{Tor}^*_{\Bbbk[v_1,\ldots,v_m]}(\Bbbk[K],\Bbbk)$ which respects products and (higher) Massey products. Then we can study Golodness of Stanley-Reisner rings by investigating the homotopy types of moment-angle complexes. Thus there is a trinity in studying Golodness of Stanley-Reisner rings consisting of algebra, combinatorics and homotopy theory. 

In this paper, we consider Problem \ref{problem} under the above trinity, and we will prove the following, where the notation in the condition (4) will be defined later. Recall that a non-empty subset of the vertex set of a simplicial complex $K$ is a minimal non-face if $N\not\in K$ and $N-i\in K$ whenever $i\in N$. Put $[m]:=\{1,\ldots,m\}$.

\begin{theorem}
\label{main}
Let $K$ be a simplicial complex on the vertex set $[m]$ such that $\Bbbk[K]$ has a minimal Taylor resolution. Then the following conditions are equivalent:
\begin{enumerate}
\item $\Bbbk[K]$ is Golod;
\item any two minimal non-faces of $K$ are not disjoint;
\item the moment-angle complex for $K$ is homotopy equivalent to a wedge of spheres;
\item for any $\underline{X}=\{X_i\}_{i=1}^m$, there is a homotopy decomposition of a polyhedral product
$$\Z_K(C\underline{X},\underline{X})\simeq\bigvee_{\emptyset\ne I\subset[m]}|\Sigma K_I|\wedge\widehat{X}^I.$$
\end{enumerate}
\end{theorem}

\begin{remark}
\begin{enumerate}
\item In Theorem \ref{main}, Golodness does not depend on the ground ring, but in general, this is not true as in \cite{K1,IK2}. We will see in the next section that minimality of the Taylor resolution of $\Bbbk[K]$ does not depend on $\Bbbk$, so in fact, Theorem \ref{main} does not depend on $\Bbbk$.
\item Recently, Frankhuize \cite{F} proved the equivalence between (1) and (2) in a more general setting by a purely algebraic manner.
\end{enumerate}
\end{remark}

Throughout this paper, let $K$ denote a simplicial complex on the vertex set $[m]$, where $K$ might have ghost vertices. 


\section{Minimality of the Taylor resolutions}

In this section, we recall the definition of the Taylor resolution for a Stanley-Reisner ring and a combinatorial characterization of its minimality due to Ayzenberg \cite{A}. We then prove the implication (1) $\Rightarrow$ (2) of Theorem \ref{main}.

Let $N_1,\ldots,N_r$ be minimal non-faces of $K$. Then we have
$$\Bbbk[K]=\Bbbk[v_1,\ldots,v_m]/(v_{N_1},\ldots,v_{N_r}).$$
The Taylor resolution for $\Bbbk[K]$ is the free $\Bbbk[v_1,\ldots,v_m]$-module resolution 
$$\cdots\xrightarrow{d}R^{-\ell}\xrightarrow{d}R^{-\ell+1}\xrightarrow{d}\cdots\xrightarrow{d}R^0=\Bbbk[v_1,\ldots,v_m]\xrightarrow{\rm proj}\Bbbk[K]$$
such that $R^{-\ell}$ is the free $\Bbbk[v_1,\ldots,v_m]$-module generated by symbols $w_{i_1,\ldots,i_\ell}$ for $1\le i_1<\cdots<i_\ell\le m$ with the differential
$$d(w_{i_1,\ldots,i_\ell})=\sum_{k=1}^\ell(-1)^{k+1}v_{N_{i_k}-N_{i_1}\cup\cdots\widehat{N_{i_k}}\cup\cdots\cup N_{i_\ell}}w_{i_1,\ldots,\widehat{i_k},\ldots,i_\ell}$$
where we set $v_\emptyset=1$. 
As usual, we say that the Taylor resolution is minimal if the differential satisfies
$$d\otimes_{\Bbbk[v_1,\ldots,v_m]}\Bbbk=0.$$
By definition, minimality of the Taylor resolution for $\Bbbk[K]$ does not depend on the ground ring $\Bbbk$, so we say that $K$ has a minimal Taylor resolution if the Taylor resolution for $\Bbbk[K]$ is minimal for some $\Bbbk$. Minimality of the Taylor resolution for $\Bbbk[K]$ can be readily translated combinatorially as: 

\begin{proposition}
[Ayzenberg \cite{A}]
\label{minimal}
Let $N_1,\ldots,N_r$ be minimal non-faces of $K$. Then $K$ has a minimal Taylor resolution if and only if
$$N_i\not\subset\bigcup_{k\ne i}N_k\qquad\text{for all }i.$$
\end{proposition}

Ayzenberg \cite{A} constructed a new simplicial complex with a minimal Taylor resolution from any given simplicial complex, and we here generalize his construction. Let $\mathtt{N}=\{N_1,\ldots,N_r\}$ be a sequence of subsets of a finite set $W$, where we allow $N_i=N_j$ for some $i\ne j$ and call $W$ the ground set of $\mathtt{N}$. By introducing new distinct points $a_1,\ldots,a_r$, we put $\widetilde{N}_i=N_i\sqcup\{a_i\}$ and $V=W\sqcup\{a_1,\ldots,a_r\}$. Define $K(\mathtt{N})$ to be the simplicial complex on the vertex set $V$ whose minimal non-faces are $\widetilde{N}_1,\ldots,\widetilde{N}_r$. Then since $\widetilde{N}_i\not\subset\bigcup_{k\ne i}\widetilde{N}_k$ for all $i$, we have the following by Proposition \ref{minimal}. 

\begin{corollary}
\label{K(N)-minimal}
$K(\mathtt{N})$ has a minimal Taylor resolution.
\end{corollary}

Notice that any simplicial complex is determined by its minimal non-faces.

\begin{proposition}
\label{K(N)}
If $K$ has a minimal Taylor resolution, then there is a sequence $\mathtt{N}$ of subsets of a finite set $W$ such that $K\cong K(\mathtt{N})$.
\end{proposition}

\begin{proof}
Let $N_1,\ldots,N_r$ be all minimal non-faces of $K$. By Proposition \ref{minimal} there exists $a_i\in N_i-\bigcup_{k\ne i}N_k$ for all $i$, where $a_1,\ldots,a_r$ are distinct. Put $W:=[m]-\{a_1,\ldots,a_r\}$. If we put $\mathtt{N}=\{N_1-a_1,\ldots,N_r-a_r\}$ which is a sequence of subsets of $W$, we have $K=K(\mathtt{N})$ as desired.
\end{proof}

We prove the implication (1) $\Rightarrow$ (2) of Theorem \ref{main}. For this, we use the following lemma, where the proof will be given in the next section. For a subset $I\subset[m]$, we put
$$K_I:=\{\sigma\in K\,\vert\,\sigma\subset I\}.$$

\begin{lemma}
\label{join}
If $K_{I\cup J}=\partial\Delta^I*\partial\Delta^J$ for some non-empty $I,J\subset[m]$ with $I\cap J=\emptyset$, then $K$ is not Golod, where $\Delta^S$ denotes the full simplex on a finite set $S$.
\end{lemma}

We here record an obvious fact of minimal non-faces, where we omit the proof. For a simplex $\sigma\in K$, let $\mathrm{lk}_K(\sigma)$ denote the link of $\sigma$ in $K$. 

\begin{lemma}
\label{min-non-face}
Let $N_1,\ldots,N_r$ be minimal non-faces of $K$.
\begin{enumerate}
\item For a simplex $\sigma\in K$, any minimal non-face of $\mathrm{lk}_K(\sigma)$ has the form $N_i-\sigma$ for some $i$.
\item For a subset $I\subset[m]$, minimal non-faces of $K_I$ are $N_i$'s with $N_i\subset I$.
\end{enumerate}
\end{lemma}

\begin{proposition}
\label{1->2}
Suppose $K$ has a minimal Taylor resolution. If $\Bbbk[K]$ is Golod, then any two minimal non-faces of $K$ are not disjoint.
\end{proposition}

\begin{proof}
Let $N_1,\ldots,N_r$ be minimal non-faces of $K$. Assume $N_i\cap N_j=\emptyset$ for some $i\ne j$. By Proposition \ref{minimal}, we have $N_k\not\subset N_i\cup N_j$ for any $k\ne i,j$. Then by Lemma \ref{min-non-face}, $N_i,N_j$ are the only minimal non-faces of $K_{N_i\cup N_j}$. It follows that $K_{N_i\cup N_j}=\partial\Delta^{N_i}*\partial\Delta^{N_j}$. Then we have $|N_i|\ge 1$ and $|N_j|\ge 1$. Thus by Lemma \ref{join}, $K$ is not Golod, completing the proof. 
\end{proof}


\section{Polyhedral products}

In this section, we recall the definition of polyhedral products and their properties that we are going to use. Let $(\underline{X},\underline{A})=\{(X_i,A_i)\}_{i\in[m]}$ be a sequence of pairs of spaces indexed by vertices of $K$. The polyhedral product $\Z_K(\underline{X},\underline{A})$ is defined by
$$\Z_K(\underline{X},\underline{A}):=\bigcup_{\sigma\in K}(\underline{X},\underline{A})^\sigma\quad(\subset X_1\times\cdots\times X_m)$$
where 
$$(\underline{X},\underline{A})^\sigma=Y_1\times\cdots\times Y_m\quad\text{for}\quad Y_i=\begin{cases}X_i&i\in\sigma\\A_i&i\not\in\sigma.\end{cases}$$ 
For a sequence of pointed spaces $\underline{X}=\{X_i\}_{i\in[m]}$, we put $(C\underline{X},\underline{X}):=\{(CX_i,X_i)\}_{i\in[m]}$, where $CY$ denotes the reduced cone of a pointed space $Y$. The real moment-angle complex $\RZ_K$ is the polyhedral product $\Z_K(C\underline{X},\underline{X})$ with $X_i=S^0$ for all $i$ while the moment-angle complex $\Z_K$ is $\Z_K(C\underline{X},\underline{X})$ with $X_i=S^1$ for any $i$. Recall from \cite{IK1} that the fat wedge filtration
$$*=\Z_K^0(C\underline{X},\underline{X})\subset\Z_K^1(C\underline{X},\underline{X})\subset\cdots\subset\Z_K^m(C\underline{X},\underline{X})=\Z_K(C\underline{X},\underline{X})$$
is defined by
$$\Z_K^i(C\underline{X},\underline{X})=\{(x_1,\ldots,x_m)\in\Z_K(C\underline{X},\underline{X})\,\vert\,\text{at least }m-i\text{ of }x_i\text{ are basepoints}\}.$$
In \cite{IK1}, the fat wedge filtration is shown to be quite useful in studying the homotopy type of a polyhedral product $\Z_K(C\underline{X},\underline{X})$. For example, it is shown that the fat wedge filtration splits after a suspension so that we can recover the homotopy decomposition of Bahri, Bendersky, Cohen and Gitler \cite{BBCG} as follows. Let $|L|$ denote the geometric realization of a simplicial complex $L$, and put $\widehat{X}^I:=\bigwedge_{i\in I}X_i$ for a sequence of pointed spaces $\underline{X}=\{X_i\}_{i\in[m]}$.

\begin{theorem}
[Iriye and Kishimoto \cite{IK1} (cf. Bahri, Bendersky, Cohen and Gitler \cite{BBCG})]
There is a homotopy decomposition
$$\Sigma\Z_K(C\underline{X},\underline{X})\simeq\Sigma\bigvee_{i=1}^m\Z_K^i(C\underline{X},\underline{X})/\Z_K^{i-1}(C\underline{X},\underline{X})=\Sigma\bigvee_{\emptyset\ne I\subset[m]}|\Sigma K_I|\wedge\widehat{X}^I.$$
\end{theorem}

We call this homotopy decomposition the BBCG decomposition. Let us consider a desuspension of the BBCG decomposition. As for the moment-angle complexes, desuspension is completely characterized as:

\begin{theorem}
[Iriye and Kishimoto \cite{IK1}]
\label{Z-BBCG}
The moment-angle complex $\Z_K$ is a suspension if and only if its BBCG decomposition desuspends.
\end{theorem}

Then as we will see in Corollary \ref{suspension-Golod} below that a desuspension of the BBCG decomposition of $\Z_K(C\underline{X},\underline{X})$ is closely related with Golodness of $\Bbbk[K]$. So we recall from \cite{IK1} a criterion for desuspending the BBCG decomposition. It is shown in \cite{IK1} that to investigate the fat wedge filtration of $\Z_K(C\underline{X},\underline{X})$, the fat wedge filtration of the real moment-angle complex $\RZ_K$ plays an important role. The fat wedge filtration of $\RZ_K$ has the following property.

\begin{theorem}
[Iriye and Kishimoto \cite{IK1}]
\label{IK-cone}
There is a map $\varphi_{K_I}\colon|K_I|\to\RZ_K^{i-1}$ for each $I\subset[m]$ with $|I|=i$ such that $\RZ_K^i$ is obtained from $\RZ_K^{i-1}$ by attaching cones by maps $\varphi_{K_I}$ for all $I\subset[m]$ with $|I|=i$.
\end{theorem}

We say that the fat wedge filtration of $\RZ_K$ is trivial if $\varphi_{K_I}$ is null homotopic for any $\emptyset\ne I\subset[m]$. Then if the fat wedge filtration of $\RZ_K$ is trivial, the BBCG decomposition for $\RZ_K$ desuspends. Moreover, we have:

\begin{theorem}
[Iriye and Kishimoto \cite{IK1}]
\label{IK-decomp}
If the fat wedge filtration of $\RZ_K$ is trivial, then the BBCG decomposition of $\Z_K(C\underline{X},\underline{X})$ desuspends for any $\underline{X}$.
\end{theorem}

We pass to the connection between Golodness and moment-angle complexes. In \cite{BBP}, Baskakov, Buchstaber and Panov observed that the cellular cochain complex with coefficient $\Bbbk$ of the natural cell structure of the moment-angle complex $\Z_K$ is isomorphic to the Koszul resolution of $\Bbbk$ over $\Bbbk[K]$ tensored with $\Bbbk[K]$. As a result, we have:

\begin{theorem}
[Baskakov, Buchstaber and Panov \cite{BBP}]
\label{BBP}
There is an isomorphism
$$H^*(\Z_K;\Bbbk)\cong\mathrm{Tor}^*_{\Bbbk[v_1,\ldots,v_m]}(\Bbbk[K],\Bbbk)$$
which respects all products and (higher) Massey products.
\end{theorem}

\begin{corollary}
\label{suspension-Golod}
If $\Z_K$ is a suspension, $\Bbbk[K]$ is Golod for any commutative ring $\Bbbk$.
\end{corollary}

Then by Theorem \ref{Z-BBCG}, we obtain:

\begin{corollary}
If the fat wedge filtration of $\RZ_K$ is trivial, then $\Bbbk[K]$ is Golod over any commutative ring $\Bbbk$.
\end{corollary}

We close this section by proving Lemma \ref{join}.

\begin{proof}
[Proof of Lemma \ref{join}]
By definition, we have $\Z_{\partial\Delta^W}=S^{2|W|-1}$ for a finite set $W$, and $\Z_{K*L}=\Z_K\times\Z_L$. Then we have $\Z_{\partial\Delta^I*\partial\Delta^J}=S^{2|I|-1}\times S^{2|J|-1}$. On the other hand, $\Z_{K_I}$ is a retract of $\Z_K$. So if $K_{I\cup J}=\partial\Delta^I*\partial\Delta^J$, the cohomology of $\Z_K$ in any coefficient has a non-trivial product, implying that $K$ is not Golod by Theorem \ref{BBP}. Thus the proof is completed.
\end{proof}


\section{Proof of Theorem \ref{main}}

We first investigate properties of simplicial complexes whose Stanley-Reisner rings have minimal Taylor resolutions. Then by Corollary \ref{K(N)-minimal} and Proposition \ref{K(N)}, we consider a simplicial complex $K(\mathtt{N})$ in Section 2. We recall notation for $K(\mathtt{N})$. $\mathtt{N}$ is a sequence $\{N_1,\ldots,N_r\}$ of subsets of a finite set $W$, and $\widetilde{N}_1,\ldots,\widetilde{N}_r$ are minimal non-faces of $K(\mathtt{N})$ such that $\widetilde{N}_i=N_i\sqcup\{a_i\}$ and $W\sqcup\{a_1,\ldots,a_r\}$ is the vertex set of $K(\mathtt{N})$. Put $m:=|W|+r$ which is the number of vertices of $K(\mathtt{N})$. For $w\in W$ we set 
$$\mathtt{N}_w:=\{N_i-w\,\vert\,i=1,\ldots,r\},\quad\widehat{\mathtt{N}}_w:=\{N_i\,\vert\,w\not\in N_i\},\quad A_w:=\{a_i\,\vert\,w\in N_i\}$$
where the ground sets of both $\mathtt{N}_w$ and $\widehat{\mathtt{N}}_w$ are $W-w$. Let $\mathrm{dl}_K(v)$ denote the deletion of a vertex $v$ in $K$. The following properties of the link and the deletion of $K(\mathtt{N})$ are immediate from Lemma \ref{min-non-face}.

\begin{lemma}
\label{dl-lk}
For $w\in W$ we have
$$\mathrm{lk}_{K(\mathtt{N})}(w)=K(\mathtt{N}_w),\quad\mathrm{dl}_{K(\mathtt{N})}(w)=K(\widehat{\mathtt{N}}_w)*\Delta^{A_w}.$$
\end{lemma}


We next describe the homotopy type of $|K(\mathtt{N})|$.

\begin{proposition}
\label{homotopy-type}
We have
$$|K(\mathtt{N})|\simeq\begin{cases}S^{|W|-1}&N_1\cup\cdots\cup N_r=W\\
*&\text{otherwise}\end{cases}$$
where we put $S^{-1}=\emptyset$. Moreover, for a sequence $\mathtt{M}=\{M_1,\ldots,M_r\}$ of subsets of $W$ satisfying $M_i\subset N_i$ for all $i$ and $M_1\cup\cdots\cup M_r=W$, the inclusion $|K(\mathtt{M})|\to|K(\mathtt{N})|$ is a homotopy equivalence.
\end{proposition}

\begin{proof}
We induct on $|W|$ to get the homotopy type of $K(\mathtt{N})$. When $|W|=0$, there is nothing to do. When $|W|=1$, we may assume $N_1=\cdots=N_s=W$ and $N_{s+1}=\cdots=N_r=\emptyset$ for some $0\le s\le r$, so 
\begin{equation}
\label{W=0}
K(\mathtt{N})=W\sqcup\Delta^{\{a_1,\ldots,a_s\}}.
\end{equation}
Hence if $s\ge 1$, or equivalently $N_1\cup\cdots\cup N_r=W$, then $|K(\mathtt{N})|\simeq S^0$, and if $s=0$, or equivalently $N_1\cup\cdots\cup N_r\ne W$, then $|K(\mathtt{N})|$ is contractible. We assume the case $m-1$ and prove the case $m$. Notice that for any $w\in W$, there is a pushout of spaces
\begin{equation}
\label{pushout}
\xymatrix{|\mathrm{lk}_{K(\mathtt{N})}(w)|\ar[r]\ar[d]&|\mathrm{lk}_{K(\mathtt{N})}(w)*w|\ar[d]\\
|\mathrm{dl}_{K(\mathtt{N})}(w)|\ar[r]&|K(\mathtt{N})|.}
\end{equation}
For $W\ne N_1\cup\cdots\cup N_r$, we take $w\in W-N_1\cup\cdots\cup N_r$. Then we have $\mathtt{N}_w=\widehat{\mathtt{N}}_w$, implying $\mathrm{lk}_{K(\mathtt{N})}(w)=\mathrm{dl}_{K(\mathtt{N})}(w)$ by Lemma \ref{dl-lk}. Then we get $|K(\mathtt{N})|=|\mathrm{lk}_{K(\mathtt{N})}(w)*w|\simeq *$. For $W=N_1\cup\cdots\cup N_r$, we take any $w\in W$, and we have $A_w\ne\emptyset$, so by Lemma \ref{dl-lk} $|\mathrm{dl}_{K(\mathtt{N})}(w)|$ is contractible. Since $|\mathrm{lk}_{K(\mathtt{N})}(w)*w|$ is also contractible, we obtain $|K(\mathtt{N})|\simeq\Sigma|\mathrm{lk}_{K(\mathtt{N})}(w)|$. By Lemma \ref{dl-lk}, we have $\mathrm{lk}_{K(\mathtt{N})}(w)=K(\mathtt{N}_w)$ to which we can apply the induction hypothesis since the ground set of $\mathtt{N}_w$ is $W-w$. Thus since $N_1\cup\cdots\cup N_r=W$ if and only if $(N_1-w)\cup\cdots\cup(N_r-w)=W-w$, we obtain the desired result.

We next prove the second assertion also by induction on $|W|$. The case $|W|=1$ follows from the identity \eqref{W=0}. Note that the diagram \eqref{pushout} is natural with respect to the canonical inclusions between $\mathtt{M},\mathtt{N}$. Then the second assertion holds by the induction hypothesis as above. 
\end{proof}

We next consider the fat wedge filtration of the real moment-angle complex $\RZ_{K(\mathtt{N})}$. We prove the following simple lemma that we are going to use.

\begin{lemma}
\label{fat-wedge}
For non-empty finite sets $A_1,\ldots,A_r$, the following hold.
\begin{enumerate}
\item $\RZ_{\partial\Delta^{A_1}*\cdots*\partial\Delta^{A_r}}=S^{|A_1|-1}\times\cdots\times S^{|A_r|-1}$.
\item Let $T$ be the fat wedge of $S^{|A_1|-1},\ldots,S^{|A_r|-1}$ , that is,
$$T:=\{(x_1,\ldots,x_r)\in S^{|A_1|-1}\times\cdots\times S^{|A_r|-1}\,\vert\,x_i\text{ is the basepoint for some }i\}.$$
Then the natural inclusion $T\to\RZ_{\partial\Delta^{A_1}*\cdots*\partial\Delta^{A_r}}^{|A_1|+\cdots+|A_r|-1}$ is a homotopy equivalence.
\end{enumerate}
\end{lemma}

\begin{proof}
(1) In general, we have $\RZ_{K*L}=\RZ_K\times\RZ_L$ for simplicial complexes $K,L$ and $\RZ_{\partial\Delta^{[m]}}=S^{m-1}$ as in the proof of Lemma \ref{join}. Thus we get the desired result.

(2) By definition $\RZ^{|A_i|-1}_{\partial\Delta^{A_i}}$ is contractible, and the inclusion $\RZ^{|A_i|-1}_{\partial\Delta^{A_i}}\to\RZ_{\partial\Delta^{A_i}}$ is a cofibration. Then since
$$\RZ_{\partial\Delta^{A_1}*\cdots*\partial\Delta^{A_r}}^{|A_1|+\cdots+|A_r|-1}=\bigcup_{i=1}^r(\RZ_{\partial\Delta^{A_1}}\times\cdots\times\RZ_{\partial\Delta^{A_i}}^{|A_i|-1}\times\cdots\times\RZ_{\partial\Delta^{A_r}}),$$
the proof is completed.
\end{proof}

We now prove triviality of the map $\varphi_{K(\mathtt{N})}\colon|K(\mathtt{N})|\to\RZ_{K(\mathtt{N})}^{m-1}$ of Theorem \ref{IK-cone} when $\widetilde{N}_i\cap\widetilde{N}_j\ne\emptyset$ for any $i,j$, that is, under the condition (2) of Theorem \ref{main}. When $N_1\cup\cdots\cup N_r\ne W$, $\varphi_{K(\mathtt{N})}$ is trivial since $|K(\mathtt{N})|$ is contractible by Proposition \ref{homotopy-type}. Then we assume $N_1\cup\cdots\cup N_r=W$. We put
$$\mathtt{M}=\{M_1,\ldots,M_r\}\quad\text{for}\quad M_i=N_i-N_1\cup\cdots\cup N_{i-1}.$$
Then we have $M_1\cup\cdots\cup M_r=W$ and $M_i\subset N_i$ for all $i$. So by Proposition \ref{homotopy-type} the inclusion $|K(\mathtt{M})|\to|K(\mathtt{N})|$ is a homotopy equivalence. Since the map $\varphi_K$ is natural with respect to inclusions of simplicial complexes by definition \cite{IK1}, there is a commutative diagram
$$\xymatrix{|K(\mathtt{M})|\ar[r]^{\varphi_{K(\mathtt{M})}}\ar[d]^{\rm incl}_\simeq&\RZ_{K(\mathtt{M})}^{m-1}\ar[d]\\
|K(\mathtt{N})|\ar[r]^{\varphi_{K(\mathtt{N})}}&\RZ_{K(\mathtt{N})}^{m-1}.}$$
Then it is sufficient to prove that the composite around the right perimeter is null homotopic.

Since $\widetilde{M}_i\cap\widetilde{M}_j=\emptyset$ for $i\ne j$, we have 
$$K(\mathtt{M})=\partial\Delta^{\widetilde{M}_1}*\cdots*\partial\Delta^{\widetilde{M}_r}\quad\text{and}\quad K(\mathtt{M})_U=\partial\Delta^{\widetilde{M}_2}*\cdots*\partial\Delta^{\widetilde{M}_r}$$
where $U=\widetilde{M}_2\cup\cdots\cup\widetilde{M}_r$. Then by Lemma \ref{fat-wedge} we get the following.

\begin{proposition}
\label{RZ-K(M)}
We have
$$\RZ_{K(\mathtt{M})}\cong S^{|M_1|}\times\cdots\times S^{|M_r|},\quad\RZ_{K(\mathtt{M})_U}\cong S^{|M_2|}\times\cdots\times S^{|M_r|}.$$
\end{proposition}
 
We now suppose $\widetilde{N}_i\cap\widetilde{N}_j\ne\emptyset$ for any $i,j$, and fix $2\le i\le r$. We define $\mathtt{M}^i$ from $\mathtt{M}$. By our supposition, there exists $w_i\in N_1\cap N_i$. Put
 $$\mathtt{M}^i=\{M_1^i,\ldots,M_r^i\}\quad\text{for}\quad M_k^i=\begin{cases}M_i\cup w_i&k=i\\M_k&k\ne i.\end{cases}$$
Then $\widetilde{M}_j^i\cap\widetilde{M}_k^i=\emptyset$ for $j\ne k$ with $j,k\ge 2$, so quite similarly to Proposition \ref{RZ-K(M)} we have
 $$\RZ_{K(\mathtt{M}^i)_{U\cup w_i}}\cong S^{|M_2^i|}\times\cdots\times S^{|M_r^i|}=S^{|M_2|}\times\cdots\times S^{|M_i|+1}\times\cdots\times S^{|M_r|}.$$
Hence the inclusion $\RZ_{K(\mathtt{M})_U}\to\RZ_{K(\mathtt{M}^i)_{U\cup w_i}}$ is identified with the inclusion
$$S^{|M_2|}\times\cdots\times S^{|M_i|}\times\cdots\times S^{|M_r|}\xrightarrow{1\times\cdots\times{\rm incl}\times\cdots\times 1}S^{|M_2|}\times\cdots\times S^{|M_i|+1}\times\cdots\times S^{|M_r|}$$
in which the $i^\text{th}$ coordinate sphere contracts up to homotopy. It follows that the inclusion $\RZ_{K(\mathtt{M})_U}\to\RZ_{K(\mathtt{M}^2)_{U\cup w_2}}\cup\cdots\cup\RZ_{K(\mathtt{M}^r)_{U\cup w_r}}$ is null homotopic by contracting each coordinate sphere. Thus since $\RZ_{K(\mathtt{M}^2)_{U\cup w_2}}\cup\cdots\cup\RZ_{K(\mathtt{M}^r)_{U\cup w_r}}\subset\RZ_{K(\mathtt{N})}^{m-1}$, we obtain:

\begin{proposition}
\label{inclusion}
If  $\widetilde{N}_i\cap\widetilde{N}_j\ne\emptyset$ for any $i,j$, then the inclusion $\RZ_{K(\mathtt{M})_U}\to\RZ_{K(\mathtt{N})}^{m-1}$ is null homotopic.
\end{proposition}

By Proposition \ref{RZ-K(M)}, we have $\RZ_{K(\mathtt{M})_U}=S^{|M_2|}\times\cdots\times S^{|M_r|}$ which we abbreviate by $P$, and let $T$ be the fat wedge of $S^{|M_1|},\ldots,S^{|M_r|}$. Then by Proposition \ref{fat-wedge}, the inclusion $T\to\RZ_{K(\mathtt{M})}^{m-1}$ is a homotopy equivalence, and by Proposition \ref{inclusion}, the inclusion $P\to\RZ_{K(\mathtt{N})}^{m-1}$ is null homotopic. Then the map $T\to\RZ_{K(\mathtt{M})}^{m-1}$ extends to a map $T\cup CP\to\RZ_{K(\mathtt{N})}^{m-1}$, and by Theorem \ref{IK-cone}, this extension satisfies a homotopy commutative diagram
$$\xymatrix{|K(\mathtt{M})|\ar[r]^{\varphi_{K(\mathtt{M})}}\ar@{=}[d]&\RZ_{K(\mathtt{M})}^{m-1}\ar[r]\ar[d]&\RZ_{K(\mathtt{M})}\ar[d]\\
|K(\mathtt{M})|\ar[r]\ar[d]&T\cup CP\ar[r]\ar[d]&\RZ_{K(\mathtt{M})}\cup CP\ar[d]\\
|K(\mathtt{N})|\ar[r]^{\varphi_{K(\mathtt{N})}}&\RZ_{K(\mathtt{N})}^{m-1}\ar[r]&\RZ_{K(\mathtt{N})}}$$
where rows are homotopy cofibrations.

Let $T'\subset P$ be the fat wedge of $S^{|M_2|},\ldots,S^{|M_r|}$. Then we have $T=(S^{|M_1|}\times T')\cup(*\times P)$, so
$$T/P=(S^{|M_1|}\times T')/(*\times T')=S^{|M_1|}\wedge(T'\sqcup *).$$
Since $T'\sqcup*$ is a retractile subcomplex of $P\sqcup*$ in the sense of James \cite{J}, the inclusion $T'\sqcup *\to P\sqcup *$ has a left homotopy inverse after a suspension. Thus the map
$$T\cup CP\simeq T/P=S^{|M_1|}\wedge(T'\sqcup *)\to S^{|M_1|}\wedge(P\sqcup *)=(S^{|M_1|}\times P)/(*\times P)\simeq\RZ_{K(\mathtt{M})}\cup CP$$
has a left homotopy inverse since $|M_1|\ge 1$. This implies that the map $|K(\mathtt{M})|\to T\cup CP$ is null homotopic, implying so is the map $\varphi_{K(\mathtt{N})}$. Therefore we have established the following.

\begin{proposition}
\label{null-homotopic}
If $\widetilde{N}_i\cap\widetilde{N}_j\ne\emptyset$ for any $i,j$, the map $\varphi_{K(\mathtt{N})}$ is null homotopic.
\end{proposition}

\begin{theorem}
\label{all-null-homotopic}
If $\widetilde{N}_i\cap\widetilde{N}_j\ne\emptyset$ for any $i,j$, then the fat wedge filtration of $\RZ_{K(\mathtt{N})}$ is trivial.
\end{theorem}

\begin{proof}
Suppose $\widetilde{N}_i\cap\widetilde{N}_j\ne\emptyset$ for any $i,j$. By Proposition \ref{null-homotopic}, it is sufficient to prove:

{\bf Claim} : For any vertex $v$ of $K(\mathtt{N})$, $\mathrm{dl}_{K(\mathtt{N})}(v)=K(\mathtt{M})*\Delta^S$ for some $S,\mathtt{M}$ such that any two elements of $\mathtt{M}$ are not disjoint, where $S$ may be empty.

We prove this claim by induction on $|W|$. When $|W|=0$, the claim is obviously true. The case $|W|=1$ follows from Proposition \ref{homotopy-type}. Suppose the claim holds for $|W|<r$, and take a vertex $v$ of $K(\mathtt{N})$. 

{\bf Case} $v\in W$: By Lemma \ref{dl-lk}, $\mathrm{dl}_{K(\mathtt{M})}(v)=K(\widehat{\mathtt{N}}_v)*\Delta^{A_v}$. By our supposition, any two elements of $\widehat{\mathtt{N}}_v$ are not disjoint. Then the claim is true for $K(\widehat{\mathtt{N}}_v)*\Delta^{A_v}$.

{\bf Case} $v\not\in W$: Since $v=a_i$ for some $i$, we have $\mathrm{dl}_{K(\mathtt{N})}(v)=K(\mathtt{M})$, where $\mathtt{M}=\{N_j\,\vert\,j\ne i\}$.  Then the claim is obviously true for $K(\mathtt{M})$. 
\end{proof}

\begin{proof}
[Proof of Theorem \ref{main}]
The implication (1) $\Rightarrow$ (2) follows from Proposition \ref{1->2}. If (2) holds, then by Theorem \ref{all-null-homotopic}, the fat wedge filtration of $\RZ_K$ is trivial. Thus by Theorem \ref{IK-decomp}, (4) holds. Moreover, by Proposition \ref{homotopy-type}, (3) holds. When (3) or (4) holds, (1) holds by Corollary \ref{suspension-Golod}. Therefore the proof is completed.
\end{proof}

\end{document}